\documentclass[twoside]{article}

\usepackage[accepted]{aistats2020}
\usepackage[round]{natbib}

\usepackage{amsmath}
\usepackage{amsfonts,amssymb}
\usepackage{mathtools}

\usepackage{epsfig}
\usepackage{graphicx}
\usepackage{amsthm}
\newtheorem{theorem}{Theorem}
\newtheorem{lemma}{Lemma}

\usepackage{algorithmic}
\usepackage{algorithm}

\begin{document}

%

%

\twocolumn[

\aistatstitle{Fast Markov chain Monte Carlo algorithms via Lie groups}

\aistatsauthor{Steve Huntsman}

\aistatsaddress{BAE Systems FAST Labs} 
]

\begin{abstract}
From basic considerations of the Lie group that preserves a target probability measure, we derive the Barker, Metropolis, and ensemble Markov chain Monte Carlo (MCMC) algorithms, as well as variants of waste-recycling Metropolis-Hastings and an altogether new MCMC algorithm. We illustrate these constructions with explicit numerical computations, and we empirically demonstrate on a spin glass that the new algorithm converges more quickly than its siblings.
\end{abstract}

\section{\label{sec:Introduction}Introduction}

The basic problem that \emph{Markov chain Monte Carlo} (MCMC) algorithms solve is to estimate expected values using a Markov chain that has the desired probability measure as its invariant measure. Originally developed to solve problems in computational physics, MCMC algorithms have since become omnipresent in statistical inference and machine learning. Indeed, it is reasonable to suggest in line with \cite{richey2010mcmc} and \cite{brooks2011mcmc} that MCMC algorithms comprise the most ubiquitous and important class of high-level numerical algorithms discovered to date.

Consequently, the literature on MCMC algorithms is vast. However, the mostly unexplored interface of MCMC algorithms and the theory of Lie groups and Lie algebras holds a surprise. As we shall see, the space of transition matrices with a given invariant measure is a \emph{monoid} that is closely related to a Lie group. Searching for elements of this monoid with closed form expressions naturally leads to the classical Barker and Metropolis MCMC samplers. Generalizing this search leads to higher-order versions of these samplers which respectively correspond to the ensemble MCMC algorithm of \cite{neal2010ensemble} and an algorithm of \cite{delmas2009does}. Further generalizing this search leads to an algorithm which we call the \emph{higher-order programming solver} and whose convergence appears to improve on the state of the art. For each of these algorithms (only treated for finite state spaces), the acceptance mechanism is specified, but the proposal mechanism is not (though one can always extend any proposal mechanism for single states to multiple states by repeated sampling).

In this paper, we first review the basics of MCMC, Lie theory, and related work in \S \ref{sec:Background}. Next, we introduce the Lie group generated by a probability measure in \S \ref{sec:Group}. Here, Lemma \ref{lem:STOalgebra} exhibits a convenient basis of the stochastic Lie algebra. Lemma \ref{lem:power} and Theorem \ref{thm:GENbasis} next yield a convenient basis of the Lie subalgebra that annihilates a target probability measure $p$. Critically, this basis only requires knowledge of $p$ up to a multiplicative factor. 
In \S \ref{sec:Monoid}, we consider a closely related monoid, and Lemma \ref{lem:pos} shows how we can analytically produce nonnegative transition matrices that leave $p$ invariant. We then exhibit the construction of the Barker and Metropolis samplers from Lie-theoretic considerations in \S \ref{sec:BarkerMetropolis}. In \S \ref{sec:Algebra}, Lemma \ref{lem:abcX} extends Lemma \ref{lem:power}, and Theorem \ref{thm:dxd} extends Theorem \ref{thm:GENbasis} in such a way as to yield generalizations of the preceding samplers that entertain multiple proposals at once. These higher-order Barker and Metropolis samplers are explicitly constructed in \S \ref{sec:HigherOrder}. We then demonstrate their behavior on a simple example of a spin glass in \S \ref{sec:Behavior}. In \S \ref{sec:LP}, Theorem \ref{thm:LP} yields multiple-proposal transition matrices that are closest in Frobenius norm to the ``ideal'' transition matrix $1p$: in this section we introduce and demonstrate the resulting higher-order programming solver. Finally, we close with remarks in \S \ref{sec:Remarks}. 
Proofs, though brief, are relegated to Appendix \ref{sec:Proofs}.



\section{\label{sec:Background}Background}

\subsection{\label{sec:MCMC}Markov chain Monte Carlo}

As mentioned in \S \ref{sec:Introduction} and \cite{bremaud1999markov}, the basic problem of MCMC is to estimate expected values of functions with respect to a probability measure $p$ that is infeasible to construct. A common instance is where $p_j \equiv \mathcal{L}_j/Z$, where it is easy to compute $\mathcal{L}$ but hard to compute the normalizing constant $Z$ due to the scale of the problem. The approach of MCMC is to construct an irreducible, ergodic Markov chain that has $p$ as its invariant measure without using global information. 

If now $X_t$ is the state of such a chain at time $t$, then in the limit we have $X_t \sim p$ for any initial condition. For $f$ suitable, $\mathbb{E}_p f(X) = \lim_{t \rightarrow \infty} \frac{1}{t} \sum_{j=1}^t f(X_j)$
even though the $X_j$ are correlated.

A MCMC algorithm typically depends on respective \emph{proposal} and \emph{acceptance} probabilities $q_{jk} := \mathbb{P}(X' = k | X_t = j)$ and $\alpha_{jk} := \mathbb{P}(X_{t+1} = k | X' = k, X_t = j)$, which yield
$P_{jk} := \mathbb{P}(X_{t+1} = k | X_t = j) = q_{jk} \alpha_{jk}$
for the elements of the chain transition matrix.


The Hastings algorithm uses a symmetric matrix $s$ to accept a proposal with probability $\alpha_{jk} = \frac{s_{jk}}{1+t_{jk}}$, where $t_{jk} := \frac{p_j q_{jk}}{p_k q_{kj}}$. This requires $s_{jk} \le 1 + \min(t_{jk},t_{kj})$. Taking $s_{jk} = 1$ yields the Barker sampler; \cite{peskun1973optimum} shows that the optimal choice $s_{jk} = 1 + \min(t_{jk},t_{kj})$ yields the Metropolis-Hastings sampler.

\subsection{\label{sec:Lie}Lie groups and Lie algebras}

For general background on Lie groups and algebras, we refer to 
\cite{onishchik1990lie} and \cite{kirillov2008lie}. Here, we briefly restate the basic concepts in the real and finite-dimensional setting. 

A \emph{Lie group} is a group that is also a manifold, and for which the group operations are smooth. The tangent space of a Lie group $G$ at the identity is a \emph{Lie algebra} that we denote by $\frak{lie}(G)$. Besides its vector space structure, this Lie algebra inherits a version of the Lie group structure through a bilinear antisymmetric \emph{bracket} $[\cdot, \cdot]$ that satisfies the \emph{Jacobi identity}
\begin{equation}
[X,[Y,Z]] + [Y,[Z,X]] + [Z,[X,Y]] = 0. \nonumber
\end{equation}
In particular, Ado's theorem implies that a real finite-dimensional Lie group is isomorphic to a subgroup of the group $GL(n, \mathbb{R})$ of invertible $n \times n$ matrices over $\mathbb{R}$. Meanwhile, the corresponding Lie algebra is isomorphic to a Lie subalgebra of real $n \times n$ matrices, for which the bracket is the usual matrix commutator: $[X,Y] := XY-YX$. In the other direction, the usual matrix exponential gives a map from a matrix Lie algebra to the corresponding Lie group that respects both the algebra and group structures.

\subsection{\label{sec:Related}Related work}

Our higher-order Barker and Metropolis samplers respectively correspond to constructions in \cite{neal2010ensemble} and \cite{delmas2009does}. Besides ensemble algorithms, \cite{robert2018mcmc} details a large body of work on accelerating MCMC algorithms by techniques such as multiple try algorithms as in \cite{liu2000multiple}, \cite{martino2018MT}, and \cite{martino2018IRSM}; or by parallelization, as in \cite{calderhead2014parallel}. 

Work by \cite{niepert2012markov}, \cite{niepert2012mcmc}, \cite{bui2013automorphism}, \cite{shariff2015symmetries}, \cite{vandenbroeck2015lifted} and \cite{anand2016contextual} dealt with accelerating MCMC algorithms by exploiting discrete symmetries that preserve the (exact or approximate) level sets of a target measure. In a related vein, ``group moves'' for MCMC algorithms were considered by \cite{liu1999parameter} and \cite{liu2000generalised}. There is also a long tradition of learning and exploiting symmetries in data representations for machine learning (see, e.g. \cite{ludtke2018abstractions}, and \cite{anselmi2019symmetry}): including for neural networks, for which see \cite{cohen2016equivariant} and \cite{cohen2018spherical}. However, to our knowledge, the present paper is the first attempt to consider continuous symmetries that preserve a target measure in the context of MCMC.

The study of Markov models on groups has been studied in considerable depth, as in \cite{saloffcoste2001random} and \cite{ceccherini2008harmonic}. However, although the idea of applying Lie theory to Markov models motivates work on the stochastic group, actual applications themselves are few and far between, with \cite{sumner2012liemarkov} serving as an exemplar. 
 
If we sacrifice analytical tractability and/or computational convenience, it is possible to consider generic MCMC algorithms that optimize some criterion over the relevant monoid. Optimal control considerations lead to algorithms such as those of \cite{suwa2010detailed}, \cite{chen2013accelerating}, \cite{bierkens2016nonreversible}, and \cite{takahashi2016detailed} that optimize convergence at the cost of reversibility/detailed balance. Alternatively, \cite{frigessi1992optimal}, \cite{pollet2004optimal}, \cite{chen2012optimal}, \cite{wu2015optimal}, and \cite{huang2018optimal} try to optimize the asymptotic variance.

\section{\label{sec:Group}The Lie group generated by a measure}

For $1 < n \in \mathbb{N}$, let $p$ be a probability measure on $[n] := \{1,\dots,n\}$. Relying on context to resolve ambiguity, write $1 = (1,\dots,1)^T \in \mathbb{R}^n$. Following \cite{johnson1985markov}, \cite{poole1995stochastic}, \cite{boukas2015lie}, and \cite{guerra2018stochastic}, define the \emph{stochastic group} 
\begin{equation}
\label{eq:STO}
STO(n) := \{P \in GL(n,\mathbb{R}): P1 = 1 \}
\end{equation}
as the stabilizer fixing $1$ on the left in $GL(n,\mathbb{R})$, and
\begin{equation}
\label{eq:GEN}
\langle p \rangle := \{P \in STO(n) : pP = p \}
\end{equation}
as the stabilizer fixing $p$ on the right: we call $\langle p \rangle$ the \emph{group generated by $p$}. $STO(n)$ and $\langle p \rangle$ are Lie groups of respective dimension $n(n-1)$ and $(n-1)^2$.

If $P \in STO(n)$ is irreducible and ergodic, then it has a unique invariant measure that we write as 
$\langle P \rangle := 1^T(P-I+11^T)^{-1}$, 
so that $pP = p$ iff $\langle p \rangle = \langle \langle P \rangle \rangle$. 
We have that $\langle p \rangle - I \subset \frak{lie} ( \langle p \rangle ) \subset \frak{lie}(STO(n))$.

For $(j,k) \in [n] \times [n-1]$, define 
\begin{equation}
\label{eq:STObasis}
e_{(j,k)} := e_j(e_k^T-e_n^T), 
\end{equation}
where $\{e_j\}_{j \in [n]}$ is the standard basis of $\mathbb{R}^n$.

\begin{lemma}
\label{lem:STOalgebra}
The matrices $\{e_{(j,k)}\}_{(j,k) \in [n] \times [n-1]}$ form a basis of $\frak{lie}(STO(n))$ and $\left[ e_{(j,k)}, e_{(\ell,m)} \right]$ equals
\begin{equation}
\label{eq:STOcommutator}
(\delta_{k\ell} - \delta_{n\ell})e_{(j,m)} - (\delta_{mj} - \delta_{nj})e_{(\ell,k)}.
\end{equation}
\end{lemma}


This basis has the obvious advantage of computationally trivial decompositions.

For $j,k \in [n-1]$, define $r_j := p_j/p_n$ and
\begin{equation}
\label{eq:GENbasis}
e_{(j,k)}^{(p)} := e_{(j,k)} - r_j e_{(n,k)} = \left ( e_j - r_j e_n \right ) (e_k^T-e_n^T).
\end{equation}
Observe that if $p_j \equiv \mathcal{L}_j/Z$, then $r_j = \mathcal{L}_j/\mathcal{L}_n$ does not depend on $Z$. This is why MCMC methods allow us to avoid computing such normalization factors, which in turn is why MCMC methods are useful. 

For future reference, we define $r := (r_1,\dots,r_{n-1},1)$ and $r^- := (r_1,\dots,r_{n-1})$.

\begin{lemma}
\label{lem:power}
For $i \in \mathbb{N}$,
\begin{equation}
\left (e_{(j,k)}^{(p)} \right )^i = \begin{cases} I, & i = 0; \\ \left ( \delta_{jk} + r_j \right )^{i-1} e_{(j,k)}^{(p)}, & i > 0. \end{cases}
\end{equation}
\end{lemma}


\begin{theorem}
\label{thm:GENbasis}
The $e_{(j,k)}^{(p)}$ form a basis for $\frak{lie} ( \langle p \rangle )$ and
\begin{equation}
\label{eq:GENcommutator}
\left[ e_{(j,k)}^{(p)}, e_{(\ell,m)}^{(p)} \right] = \left ( \delta_{k\ell} + r_\ell \right ) e_{(j,m)}^{(p)} - \left ( \delta_{mj} + r_j \right ) e_{(\ell,k)}^{(p)}.
\end{equation}
\end{theorem}


For later convenience, we write 
\begin{equation}
f_{(j,k)}^{(p)}(t) := \frac{e^{-t(\delta_{jk} + r_j)} - 1}{\delta_{jk} + r_j}. \nonumber
\end{equation}

\section{\label{sec:Monoid}The positive monoid of a measure}


Most of the elements of $STO(n)$ are not \emph{bona fide} stochastic matrices because they have negative entries; meanwhile, stochastic matrices need not be invertible. We therefore consider the \emph{monoids} (i.e., semigroups with identity; cf. \cite{hilgert1993lie})
\begin{equation}
STO^+(n) := \{P \in M(n,\mathbb{R}): P1 = 1 \text{ and } P \ge 0 \},
\end{equation}
where $P \ge 0$ is interpreted per entry, and
\begin{equation}
\langle p \rangle^+ := \{P \in STO^+(n) : pP = p \}.
\end{equation}
Note that $STO^+(n) \not \subset STO(n)$ and $\langle p \rangle^+ \not \subset \langle p \rangle$ owing to the noninvertible elements on the LHSs. Also, $STO^+(n)$ and $\langle p \rangle^+$ are bounded convex polytopes.

\begin{lemma}
\label{lem:pos}
If $t_j \ge 0$, then $\exp \left ( -\sum_j t_j e_{(j,j)}^{(p)} \right ) \in \langle p \rangle^+$.
\end{lemma}


In particular, for $t \ge 0$ we have that 
\begin{equation}
\label{eq:explicitsemi1}
\exp \left (-t e_{(j,j)}^{(p)} \right ) = I + f_{(j,j)}^{(p)}(t) \cdot e_{(j,j)}^{(p)} \in \langle p \rangle^+.
\end{equation}

Unfortunately, aside from \eqref{eq:explicitsemi1}, Lemma \ref{lem:pos} does not give a way to construct explicit elements of $\langle p \rangle^+$ in closed form, or even algorithmically. This situation is an analogue of the quantum compilation problem (see \cite{dawson2005solovaykitaev}), which is by no means trivial. 

Indeed, even if the sum in the lemma's statement has only two terms, we are immediately confronted with the formidable Zassenhaus formula (see \cite{casas2012zassenhaus}).
While $\exp \left ( -t_{(j,k)} e_{(j,k)}^{(p)} - t_{(\ell,m)} e_{(\ell,m)}^{(p)} \right )$ can be evaluated in closed form with a computer algebra package, the results involve many pages of arithmetic for the case corresponding to Lemma \ref{lem:pos}, and the other possibilities all yield some negative entries.


\section{\label{sec:BarkerMetropolis}Barker and Metropolis samplers}

Despite the weak foothold that Lemma \ref{lem:pos} affords for explicit analytical constructions, we can still use \eqref{eq:explicitsemi1} to produce a MCMC algorithm parametrized by $t$. We use a simple trick of relabeling the current state as $n$ and then reversing the relabeling, so that the transition $n \rightarrow j$ becomes generic. For $P = \exp \left (-t e_{(j,j)}^{(p)} \right )$, we have $P_{jj} = 1+f_{(j,j)}^{(p)}(t)$, $P_{jn} = -f_{(j,j)}^{(p)}(t)$, $P_{nj} = -f_{(j,j)}^{(p)}(t)r_j$, and $P_{nn} = 1+f_{(j,j)}^{(p)}(t)r_j$. In particular, 
\begin{equation}
\frac{P_{jn}}{P_{nj}} = \frac{1}{r_j} = \frac{p_n}{p_j}. \nonumber
\end{equation}
That is, detailed balance is automatic.

From the point of view of convergence, the optimal value for $t$ is the one that maximizes the off-diagonal terms, i.e., $t = \infty$. Here we get $P_{jj} = \frac{r_j}{1+r_j}$, $P_{jn} = \frac{1}{1+r_j}$, $P_{nj} = \frac{r_j}{1+r_j}$, and $P_{nn} = \frac{1}{1+r_j}$. The corresponding MCMC algorithm is a Barker sampler. 

In light of \eqref{eq:explicitsemi1}, we can improve on the Barker sampler almost trivially. We have that $I - \tau e_{(j,j)}^{(p)} \in \langle p \rangle^+$ iff $0 \le \tau \le \min(1,r_j^{-1})$. But 
\begin{equation}
\label{eq:metropolis}
\left ( I - \min(1,r_j^{-1}) \cdot e_{(j,j)}^{(p)} \right )_{nj} = \min(1,r_j)
\end{equation}
is just the Metropolis acceptance ratio. That is, we have derived the Barker and Metropolis samplers from basic considerations of symmetry and (in the latter case) optimality. 

\begin{algorithm}[tb]
   \caption{Metropolis}
   \label{alg:metropolis}
\begin{algorithmic}
   \STATE {\bfseries Input:} Runtime $T$ and oracle for $r$
   \STATE Initialize $t=0$ and $X_0$
   \REPEAT
   \STATE Relabel states so that $X_t = n$
   \STATE Propose $j \in [n-1]$
   \STATE Accept $X_{t+1} = j$ with probability \eqref{eq:metropolis}
   \STATE Undo relabeling; set $t = t+1$
   \UNTIL{$t = T$}
   \STATE {\bfseries Output:} $\{X_t\}_{t=0}^T \sim p^{\times (T+1)}$ (approximately)
   \end{algorithmic}
\end{algorithm}

Note that the proposal mechanism that selects the state $j$ is unspecified and unconstrained by our construction (i.e., $\mathcal{J}$ can be drawn from an arbitrary joint distribution on the subset of $[n-1]^d$ without duplicate entries). That is, our approach separates concerns between the proposal and acceptance mechanisms, and only focuses on the latter. In later sections of this paper, the proposal mechanism that selects a set of states will similarly be unspecified and unconstrained. However, in \S \ref{sec:Behavior} we select the elements of proposal sets uniformly at random without replacement for illustrative purposes. That said, a good proposal mechanism is of paramount importance for MCMC algorithms.

\section{\label{sec:Algebra}Some algebra}

The Barker and Metropolis samplers can be regarded as among the very ``simplest'' MCMC methods in the sense that \eqref{eq:explicitsemi1} is among the very sparsest possible nontrivial matrices in $\langle p \rangle^+$. This suggests the question: what happens if we are willing to sacrifice some sparsity? In other words, what if we consider possible transitions to more than one state? It is natural to expect both better convergence and increased complexity. The (utterly impractical and degenerate) limiting case is the matrix $1p$, and the practical starting case is the Barker and Metropolis samplers. Meanwhile, it is also natural to wonder how (or if) we can analytically construct more general elements of $\langle p \rangle^+$ than \eqref{eq:explicitsemi1}.

The following generalization of Lemma \ref{lem:power} is the first step toward an answer to the preceding questions. For $\mathcal{J} := \{j_1,\dots,j_d\} \subseteq [n-1]$ and a matrix $\alpha \in M_{n-1,n-1}$, define $\alpha_{(\mathcal{J})} \in M_{d,d}$ by $(\alpha_{(\mathcal{J})})_{uv} := \alpha_{j_u j_v}$, $\alpha_{(\mathcal{J})}^{(p)} := \sum_{u,v = 1}^d \alpha_{j_u j_v} e_{(j_u,j_v)}^{(p)} \in \mathfrak{lie}(\langle p \rangle)$, and $r_{(\mathcal{J})} := (r_{j_1}, \dots, r_{j_d})$.
\begin{lemma}
\label{lem:abcX} Let $\mathcal{J} := \{j_1,\dots,j_d\} \subseteq [n-1]$. If $\gamma_{(\mathcal{J})}^{(p)} = \alpha_{(\mathcal{J})}^{(p)} \beta_{(\mathcal{J})}^{(p)}$, 
then
\begin{equation}
\label{eq:alphabeta}
\gamma_{(\mathcal{J})} = \alpha_{(\mathcal{J})} (I+1r_{(\mathcal{J})}) \beta_{(\mathcal{J})}.
\end{equation}
\end{lemma}


We remark that introducing heavy notation for Lemma \ref{lem:abcX} is worth it: the proof takes just three lines, whereas the case $d = 2$ takes about a page of algebra to check otherwise. 
Using Lemma \ref{lem:abcX}, we can readily construct an analytically convenient matrix in $\frak{lie} ( \langle p \rangle )$. 

\begin{theorem}
\label{thm:dxd}
Let $\mathcal{J} := \{j_1,\dots,j_d\} \subseteq [n-1]$, $\omega \in \mathbb{R}$ and
\begin{equation}
\label{eq:dxdA}
A_{(\mathcal{J})}^{(p;\omega)} := \omega \sum_{u,v} \left ( \delta_{j_u j_v} - \frac{1}{1+r_{(\mathcal{J})}1} r_{j_v} \right ) e_{(j_u,j_v)}^{(p)}.
\end{equation}
Then
\begin{equation}
\label{eq:dxd}
\exp tA_{(\mathcal{J})}^{(p;\omega)} = I + \frac{e^{\omega t}-1}{\omega}A_{(\mathcal{J})}^{(p;\omega)}.
\end{equation}
Moreover, $\exp \left ( -tA_{(\mathcal{J})}^{(p;\omega)} \right ) \in \langle p \rangle^+ \cap GL(n, \mathbb{R})$ if $t \ge 0$. In particular, the \emph{Barker matrix}
\begin{equation}
\label{eq:BpJ}
\mathcal{B}_{(\mathcal{J})}^{(p)} := I - \omega^{-1} A_{(\mathcal{J})}^{(p;\omega)}
\end{equation}
is in $\langle p \rangle^+$, and does not depend on $\omega$.
\end{theorem}


Let $\Delta$ denote the map that takes a matrix to the vector of its diagonal entries, and indicate the boundary of a set using $\partial$.
\begin{lemma}
\label{lem:metropolis}
The \emph{Metropolis matrix}
\begin{equation}
\label{eq:MpJ}
\mathcal{M}_{(\mathcal{J})}^{(p)} := I - \frac{1}{\max \Delta \left ( A_{(\mathcal{J})}^{(p;\omega)} \right )}A_{(\mathcal{J})}^{(p;\omega)}
\end{equation} 
is in $\partial \langle p \rangle^+$ and does not depend on $\omega$.
\end{lemma}


\subsection{\label{sec:Example}Example}

    %
    %
    %
    %
    %
    
As an example, consider $p = (1,2,3,4,10)/20$ and $\mathcal{J} = \{1,2,3\}$. Now \eqref{eq:dxdA} is given by
\begin{equation}
A_{(\mathcal{J})}^{(p;\omega)} = 
\frac{\omega}{16}\begin{psmallmatrix} 15 & -2 & -3 & 0 & -10 \\ -1 & 14 & -3 & 0 & -10 \\ -1 & -2 & 13 & 0 & -10 \\ 0 & 0 & 0 & 0 & 0 \\ -1 & -2 & -3 & 0 & 6 \end{psmallmatrix}. \nonumber
\end{equation}
For $\omega = 1$ and $t = -\log 2$, \eqref{eq:dxd} is given by 
\begin{equation}
\exp \left ( \log 2 \cdot A_{(\mathcal{J})}^{(p;1)} \right ) = 
\frac{1}{32}\begin{psmallmatrix} 17 & 2 & 3 & 0 & 10 \\ 1 & 18 & 3 & 0 & 10 \\ 1 & 2 & 19 & 0 & 10 \\ 0 & 0 & 0 & 32 & 0 \\ 1 & 2 & 3 & 0 & 26 \end{psmallmatrix}. \nonumber
\end{equation}
Finally, \eqref{eq:BpJ} and \eqref{eq:MpJ} are respectively given by
\begin{equation}
\mathcal{B}_{(\mathcal{J})}^{(p)} =
\frac{1}{16}\begin{psmallmatrix} 1 & 2 & 3 & 0 & 10 \\ 1 & 2 & 3 & 0 & 10 \\ 1 & 2 & 3 & 0 & 10 \\ 0 & 0 & 0 & 16 & 0 \\ 1 & 2 & 3 & 0 & 10 \end{psmallmatrix}; \quad \mathcal{M}_{(\mathcal{J})}^{(p)} =
\frac{1}{15}\begin{psmallmatrix} 0 & 2 & 3 & 0 & 10 \\ 1 & 1 & 3 & 0 & 10 \\ 1 & 2 & 2 & 0 & 10 \\ 0 & 0 & 0 & 15 & 0 \\ 1 & 2 & 3 & 0 & 9 \end{psmallmatrix}. \nonumber
\end{equation}

\section{\label{sec:HigherOrder}Higher-order samplers}

The idea now is to let $n \rightarrow j \in \mathcal{J}$ correspond to a generic transition as in \S \ref{sec:BarkerMetropolis}. (Again, we do not specify or constrain a proposal that produces $\mathcal{J}$.) This yields novel MCMC algorithms using \eqref{eq:BpJ} and \eqref{eq:MpJ} which we respectively call \emph{higher-order Barker and Metropolis samplers} and abbreviate as HOBS and HOMS. 

The corresponding matrix elements are readily obtained with a bit of arithmetic: we have that
\begin{align}
\label{eq:entriesA}
\frac{1}{\omega} \left ( A_{(\mathcal{J})}^{(p;\omega)} \right )_{j_u j_u} & = 1 - \frac{r_{j_u}}{1+r_{(\mathcal{J})}1}; \nonumber \\
\frac{1}{\omega} \left ( A_{(\mathcal{J})}^{(p;\omega)} \right )_{n j_u} & = - \frac{r_{j_u}}{1+r_{(\mathcal{J})}1}; \nonumber \\
\frac{1}{\omega} \left ( A_{(\mathcal{J})}^{(p;\omega)} \right )_{n n} & = \frac{r_{(\mathcal{J})}1}{1+r_{(\mathcal{J})}1},
\end{align}
which yields the HOBS:
\begin{align}
\label{eq:HOBS}
\left ( \mathcal{B}_{(\mathcal{J})}^{(p)} \right )_{n j_u} & = \frac{r_{j_u}}{1+r_{(\mathcal{J})}1}; \quad
\left ( \mathcal{B}_{(\mathcal{J})}^{(p)} \right )_{n n} & = \frac{1}{1+r_{(\mathcal{J})}1}.
\end{align}

Meanwhile, 
\begin{equation}
\frac{1}{\omega} \max \Delta \left ( A_{(\mathcal{J})}^{(p;\omega)} \right ) = \frac{1+r_{(\mathcal{J})}1-\min \{ 1, \min r_{(\mathcal{J})} \}}{1+r_{(\mathcal{J})}1} \nonumber
\end{equation}
which yields the HOMS:
\begin{align}
\label{eq:HOMS}
\left ( \mathcal{M}_{(\mathcal{J})}^{(p)} \right )_{n j_u} & = \frac{r_{j_u}}{1+r_{(\mathcal{J})}1-\min \{ 1, \min r_{(\mathcal{J})} \}}; \nonumber \\
\left ( \mathcal{M}_{(\mathcal{J})}^{(p)} \right )_{n n} & = 1 - \frac{r_{(\mathcal{J})}1}{1+r_{(\mathcal{J})}1-\min \{ 1, \min r_{(\mathcal{J})} \}}.
\end{align}

\begin{algorithm}[tb]
   \caption{HOMS}
   \label{alg:homs}
\begin{algorithmic}
   \STATE {\bfseries Input:} Runtime $T$ and oracle for $r$
   \STATE Initialize $t=0$ and $X_0$
   \REPEAT
   \STATE Relabel states so that $X_t = n$
   \STATE Propose $\mathcal{J} = \{j_1,\dots,j_d\} \subseteq [n-1]$
   \STATE Accept $X_{t+1} = j_u$ with probability \eqref{eq:HOMS}
   \STATE Undo relabeling; set $t = t+1$
   \UNTIL{$t = T$}
   \STATE {\bfseries Output:} $\{X_t\}_{t=0}^T \sim p^{\times (T+1)}$ (approximately)
   \end{algorithmic}
\end{algorithm}

The HOBS turns out to be equivalent to the ensemble MCMC algorithm of \cite{neal2010ensemble} as described in \cite{martino2018MT} and \cite{martino2018IRSM}. While in \S \ref{sec:Behavior} the proposal mechanism we use for the HOBS essentially (apart from non-replacement, which technically induces jointness) amounts to the \emph{independent} ensemble MCMC sampler, in general this is not the case. A more sophisticated proposal mechanism with joint structure will be more powerful. That said, we reiterate that our approach is agnostic with respect to the details of proposals.

On the other hand, the HOMS is different than a \emph{multiple-try Metropolis sampler} (MTMS), including the independent MTMS described in \cite{martino2018MT}. In the HOMS, we sample from $\mathcal{J} \cup \{n\}$ to perform a state transition in a single step according to \eqref{eq:HOMS}, whereas a MTMS first samples from $\mathcal{J}$ and then accepts or rejects the result.
The HOMS (and for that matter, also the HOBS) turns out to be a slightly special case of a construction in \S 2.3 of \cite{delmas2009does}. This work uses a ``proposition kernel''  that is defined by assigning a probability distribution on the power set of the state space to each element of the state space. Roughly speaking, the HOMS and HOBS result if this distribution is independent of the individual element (i.e., varying only with the subset).

\section{\label{sec:Behavior}Behavior}

As $d = |\mathcal{J}|$ increases and/or $p$ becomes less uniform (e.g., in a low-temperature limit), the difference between the HOBS and HOMS decreases, since in either limit we have $\min \{ 1, \min r_{(\mathcal{J})} \} \ll 1+r_{(\mathcal{J})}1$. Although these limits are where one might hope to gain the most utility from improved MCMC algorithms, the HOMS can still provide an advantage in, e.g. the high-temperature part of a parallel tempering scheme (see \cite{earl2005tempering}) or for $d > 1$ but small, with elements chosen in complementary ways (uniformly at random, near current/previous states, etc.).

We exhibit the the behavior of the HOBS and HOMS on a Sherrington-Kirkpatrick (SK) spin glass in Figure \ref{fig:SK_9spins_b025}. 
As \cite{bolthausen2007spin} and \cite{panchenko2012sk} remark, the SK spin glass is the distribution \begin{equation}
\label{eq:SK}
p(s) := \textstyle{Z^{-1} \exp \left ( -\frac{\beta}{\sqrt N} \sum_{jk} J_{jk} s_j s_k \right )}
\end{equation}
over spins $s \in \{\pm 1\}^N$, where $J$ is a symmetric $N \times N$ matrix with IID standard Gaussian entries and $\beta$ is the inverse temperature. 

The disordered landscape of the SK model suits a straightforward evaluation of higher-order samplers: more detailed benchmarks seem to require specific assumptions (e.g., exploiting the particular form of a spin Hamiltonian for Swendsen-Wang updates) and/or parameters (e.g., the choice of additional temperatures for parallel tempering or of a vorticity matrix for non-reversible Metropolis-Hastings). In particular, we do not consider sophisticated or diverse ways to generate elements of proposal sets $\mathcal{J}$: instead, we merely select the elements of $\mathcal{J}$ uniformly at random without replacement. We also use the same pseudorandom number generator initial state for the HOBS and HOMS simulations in order to highlight their relative behavior, and pick $\beta$ low enough ($1/4$ and $1$) so that the behavior of a single run is sufficiently representative to make qualitative judgments.

\begin{figure}[htbp]
\includegraphics[trim = 10mm 0mm 10mm 0mm, clip, width=\columnwidth,keepaspectratio]{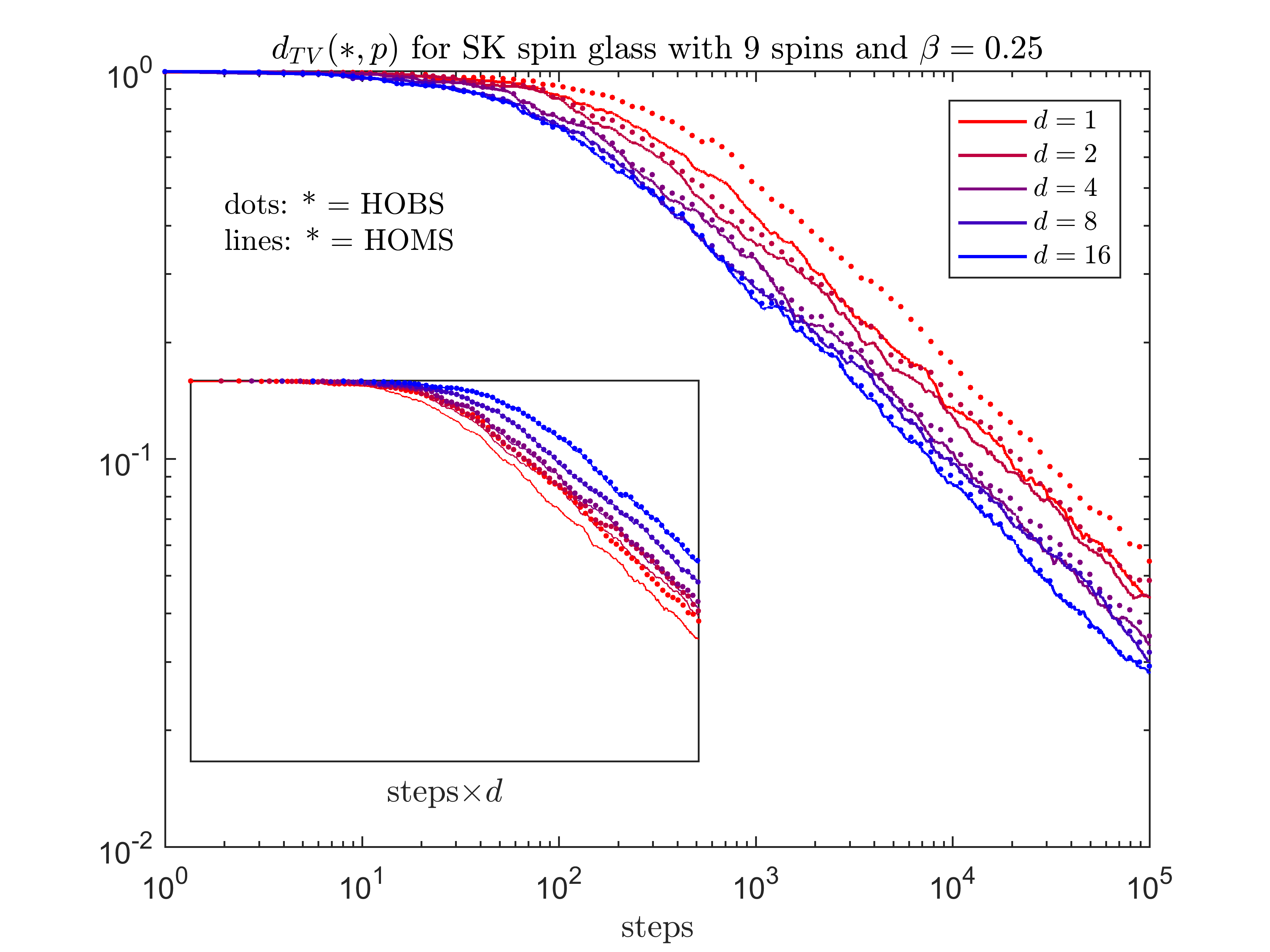}
\caption{ \label{fig:SK_9spins_b025} Total variation distance between the HOBS/HOMS with proposal sets $\mathcal{J}$ (elements distributed uniformly without replacement) of varying sizes $d$ and \eqref{eq:SK} with 9 spins and $\beta = 1/4$. Inset: the same data and window, with horizontal axis normalized by $d$. Not shown: for $\beta = 1$, the HOMS visibly outperforms the HOBS, but only for $d = 1$.
} 
\end{figure} %

%

The inset figure shows that although higher-order samplers converge more quickly, this requires more evaluations of probability ratios. Parallelism is therefore necessary to make higher-order samplers worthwhile.

As noted above, the HOMS gives results very close to the HOBS except for small values of $d$ or a more uniform target distribution $p$. Increasing the number of spins and/or considering an Edwards-Anderson spin glass also gives qualitatively similar results.

\section{\label{sec:LP}A linear program}

We can push these ideas further by using a linear program to (implicitly) construct transition matrices with the desired invariant measure and that are optimal in some sense, though the regime of utility then narrows to situations where computing likelihoods is hard enough and parallel resources are sufficient to justify the added computational costs. For example, the particular objective function
$-1_\mathcal{J}^T \tau_{(\mathcal{J})}^{(p)} r_\mathcal{J}^T$ considered immediately after \eqref{eq:genericObjective} yields an optimal sparse approximation of the ``ultimate'' transition matrix $1p$. (To the best of our knowledge, this construction has not been considered elsewhere.)

Toward this end, define $1_\mathcal{J} \in \mathbb{R}^n$ by 
\begin{equation}
(1_\mathcal{J})_j := \begin{cases} 1 & \text{if } j \in \mathcal{J} \cup \{n\} \\ 0 & \text{otherwise}, \end{cases} \nonumber
\end{equation}
$1_\mathcal{J}^- := ((1_\mathcal{J})_1,\dots,(1_\mathcal{J})_{n-1})^T$, $r_\mathcal{J} := r \odot 1_\mathcal{J}^T$, and $r_\mathcal{J}^- := r^- \odot (1_\mathcal{J}^-)^T$, where $\odot$ is the entrywise or Hadamard product (note that $r_\mathcal{J} \in \mathbb{R}^n$, while $r_{(\mathcal{J})} \in \mathbb{R}^{|\mathcal{J}|}$ has been defined previously). 

Writing $\Delta$ for the matrix diagonal map, using the notation of Lemma \ref{lem:abcX}, and noting that
\begin{equation}
\label{eq:tau}
\tau_{(\mathcal{J})}^{(p)} = \begin{pmatrix} I_{n-1} \\ -r_\mathcal{J}^- \end{pmatrix} \tau \begin{pmatrix} I_{n-1} & -1_\mathcal{J}^- \end{pmatrix},
\end{equation}
we have that $I - \tau_{(\mathcal{J})}^{(p)} \in \langle p \rangle^+$ iff
\begin{subequations}
\label{eq:constraints}
\begin{align}
0 & \le I_{n-1} - \Delta(1_\mathcal{J}^-) \tau \Delta(1_\mathcal{J}^-) \le 1; \label{eq:bodyConstraint} \\
0 & \le \tau 1_\mathcal{J}^- \le 1; \label{eq:lastColConstraint} \\
0 & \le r_\mathcal{J}^- \tau \le 1; \label{eq:lastRowConstraint} \\
0 & \le r_\mathcal{J}^- \tau 1_\mathcal{J}^- \le 1. \label{eq:lastElementConstraint}
\end{align}
\end{subequations}
Constraints \eqref{eq:lastColConstraint}-\eqref{eq:lastElementConstraint} respectively force the first $n-1$ entries of the last column, the first $n-1$ entries of the last row, and the bottom right matrix entry of $\tau_{(\mathcal{J})}^{(p)}$ to be in the unit interval; \eqref{eq:bodyConstraint} forces the relevant entries of the ``coefficient matrix'' $\tau$ (as an upper left submatrix of $\tau_{(\mathcal{J})}^{(p)}$) to be in the unit interval.

Furthermore, it is convenient to set to zero the irrelevant/unspecified rows and columns of $\tau$ that do not contribute to $\tau_{(\mathcal{J})}^{(p)}$ via the constraints
\begin{equation}
\label{eq:sparsityConstraint}
\Delta(1-1_\mathcal{J}^-) \tau = \tau \Delta(1-1_\mathcal{J}^-) = 0.
\end{equation}
If we impose \eqref{eq:sparsityConstraint}, then \eqref{eq:bodyConstraint} can be replaced with
\begin{equation}
\label{eq:bodyConstraint2}
0 \le I_{n-1} - \tau \le 1.
\end{equation}

The ``diagonal'' case corresponding to Lemma \ref{lem:pos} shows that the constraints \eqref{eq:constraints} and \eqref{eq:sparsityConstraint} jointly have nontrivial solutions. It is therefore natural to consider suitable objectives and the corresponding linear programs for optimizing the MCMC transition matrix $I - \tau_{(\mathcal{J})}^{(p)}$. Toward this end, we introduce the \emph{vectorization} map $\text{vec}$ that sends a matrix to the vector obtained by stacking the matrix columns in order, and which obeys the useful identity $\text{vec}(XYZ^T) = (Z \otimes X) \text{vec}(Y)$, where $\otimes$ denotes the tensor product.

A reasonably generic objective to maximize is
\begin{equation}
\label{eq:genericObjective}
 x^T \tau_{(\mathcal{J})}^{(p)} y = (y^T \otimes x^T) \text{vec} \left ( \tau_{(\mathcal{J})}^{(p)} \right )
\end{equation}
for suitable vectors $x$ and $y$. In practice, we shall take $x = 1_\mathcal{J}$ and $y = -r_\mathcal{J}^T$, so that our objective maximizes the Frobenius inner product of $I - \tau_{(\mathcal{J})}^{(p)}$ and $1_\mathcal{J} r_\mathcal{J}$ as a consequence of the equality
\begin{equation}
\text{tr} \left ( \left ( I - \tau_{(\mathcal{J})}^{(p)} \right )^T 1_\mathcal{J} r_\mathcal{J} \right ) = r_\mathcal{J} 1_\mathcal{J} - 1_\mathcal{J}^T \tau_{(\mathcal{J})}^{(p)} r_\mathcal{J}^T. \nonumber
\end{equation}
Alternatives such as $x = e_n, y = e_n$ (which discourages self-transitions) can result in convergence that slows catastrophically as $d = |\mathcal{J}|$ increases, because high-probability states are less likely to remain occupied. More surprisingly, the same sort of slowing down occurs for $x = e_n, y = -r_\mathcal{J}^T$ and even for variations upon the $n$th component of $y$: we suspect that the cause is the same, though mediated indirectly through an objective that ``overfits'' the proposed transition probabilities to the detriment of remaining in place (or in some cases ``underfits'' by yielding the identity matrix). In general, it appears nontrivial to select better choices for $x$ and $y$ than our defaults.

By \eqref{eq:tau} we get
\begin{equation}
\text{vec} \left ( \tau_{(\mathcal{J})}^{(p)} \right ) = \left [ \begin{pmatrix} I_{n-1} \\ -(1_\mathcal{J}^-)^T \end{pmatrix} \otimes \begin{pmatrix} I_{n-1} \\ -r_\mathcal{J}^- \end{pmatrix} \right ] \text{vec}(\tau),
\end{equation}
and in turn $(y^T \otimes x^T) \text{vec} \left ( \tau_{(\mathcal{J})}^{(p)} \right )$ equals
\begin{equation}
\label{eq:tensorObjective}
\left [ y^T \begin{pmatrix} I_{n-1} \\ -(1_\mathcal{J}^-)^T \end{pmatrix} \otimes x^T \begin{pmatrix} I_{n-1} \\ -r_\mathcal{J}^- \end{pmatrix} \right ] \text{vec}(\tau).
\end{equation}

At this point both the constraints and the objective of the linear program are explicitly specified in terms of the ``coefficient'' matrix $\tau$. We can rephrase the constraints into a more computationally convenient form, respectively rephrasing \eqref{eq:lastColConstraint}-\eqref{eq:lastElementConstraint}, \eqref{eq:sparsityConstraint}, and \eqref{eq:bodyConstraint2} as 
\begin{equation}
0 \le \begin{pmatrix} \left ( 1_\mathcal{J}^- \right )^T \otimes I_{n-1} \\ I_{n-1} \otimes r_\mathcal{J}^- \\ \left ( 1_\mathcal{J}^- \right )^T \otimes r_\mathcal{J}^- \end{pmatrix} \text{vec}(\tau) \le 1,
\end{equation}
\begin{equation}
\begin{pmatrix} I_{n-1} \otimes \Delta(1-1_\mathcal{J}^-) \\ \Delta(1-1_\mathcal{J}^-) \otimes I_{n-1} \end{pmatrix} \text{vec}(\tau) = 0,
\end{equation}
\begin{equation}
\text{vec}(I_{n-1})-1 \le \text{vec}(\tau) \le \text{vec}(I_{n-1}).
\end{equation}

Therefore, writing
\begin{align}
U_{(\mathcal{J})}^{(p)} := & \ \begin{pmatrix} I_{2n-1} \\ -I_{2n-1} \end{pmatrix} \begin{pmatrix} \left ( 1_\mathcal{J}^- \right )^T \otimes I_{n-1} \\ I_{n-1} \otimes r_\mathcal{J}^- \\ \left ( 1_\mathcal{J}^- \right )^T \otimes r_\mathcal{J}^- \end{pmatrix}; \nonumber \\
v := & \ \begin{pmatrix} 1_{2n-1} \\ 0_{2n-1} \end{pmatrix}; \nonumber \\
w_{(\mathcal{J})}^{(p)} := & \ -y^T \begin{pmatrix} I_{n-1} \\ -(1_\mathcal{J}^-)^T \end{pmatrix} \otimes x^T \begin{pmatrix} I_{n-1} \\ -r_\mathcal{J}^- \end{pmatrix}, \nonumber
\end{align}
and 
\begin{equation}
U_{(\mathcal{J})}^{(0)} := \begin{pmatrix} I_{n-1} \otimes \Delta(1-1_\mathcal{J}^-) \\ \Delta(1-1_\mathcal{J}^-) \otimes I_{n-1} \end{pmatrix},
\end{equation}
we can at last write the desired linear program (noting the inclusion of a minus sign in $w_{(\mathcal{J})}^{(p)}$ and a minimization versus a maximization as a result) in the MATLAB-ready form
\begin{subequations}
\label{eq:LP}
\begin{align}
\min_{\tau} w_{(\mathcal{J})}^{(p)} \text{vec}(\tau) \quad \text{s.t.}
	& \nonumber \\
U_{(\mathcal{J})}^{(p)} \text{vec}(\tau) \quad \le 
	& \quad v; \label{eq:LPineq} \\
U_{(\mathcal{J})}^{(0)} \text{vec}(\tau) \quad = 
	& \quad 0; \label{eq:LPeq} \\
\text{vec}(\tau) \quad \ge 
	& \quad \text{vec}(I_{n-1})-1; \label{eq:LPlb} \\
\text{vec}(\tau) \quad \le 
	& \quad \text{vec}(I_{n-1}). \label{eq:LPub}
\end{align}
\end{subequations}

As a result of the preceding discussion, we have
\begin{theorem}
\label{thm:LP}
For any $x, y \in \mathbb{R}^n$, the linear program \eqref{eq:LP} has a solution in $\langle p \rangle^+$. \qed
\end{theorem}

\subsection{\label{sec:LPexample}Example}


As in \S \ref{sec:Example}, consider $p = (1,2,3,4,10)/20$ and $\mathcal{J} = \{1,2,3\}$. The solution of the linear program with $x = 1_\mathcal{J}$ and $y = -r_\mathcal{J}^T$ yields the following element of $\langle p \rangle^+$:
\begin{equation}
\begin{psmallmatrix} 0 & 0 & 0 & 0 & 1 \\ 0 & 0 & 0 & 0 & 1 \\ 0 & 0 & 0 & 0 & 1 \\ 0 & 0 & 0 & 1 & 0 \\ 0.1 & 0.2 & 0.3 & 0 & 0.4 \end{psmallmatrix}. \nonumber
\end{equation}
For comparison, we recall that the last row of $\mathcal{M}_{(\mathcal{J})}^{(p)}$ equals 
$(0.0 \bar 6, 0.1 \bar 3, 0.2, 0, 0.6)$.

\subsection{\label{sec:HOPS}The higher-order programming sampler}

Call the sampler obtained from \eqref{eq:genericObjective} and \eqref{eq:LP} with $x = -1_\mathcal{J}$ and $y = r_\mathcal{J}^T$ the \emph{higher-order programming sampler} (HOPS). In figures \ref{fig:SK_9spins_b025_LPcomparison} and \ref{fig:SK_9spins_b100_LPcomparison}, we compare the HOMS and HOPS (cf. Figure \ref{fig:SK_9spins_b025}). It is clear from the figures that the HOPS improves upon the HOMS, which in turn improves upon the HOBS.

\begin{algorithm}[tb]
   \caption{HOPS}
   \label{alg:hops}
\begin{algorithmic}
   \STATE {\bfseries Input:} Runtime $T$ and oracle for $r$
   \STATE Initialize $t=0$ and $X_0$
   \REPEAT
   \STATE Relabel states so that $X_t = n$
   \STATE Propose $\mathcal{J} = \{j_1,\dots,j_d\} \subseteq [n-1]$
   \STATE Compute $\tau$ solving \eqref{eq:LP} with $x = 1_\mathcal{J}$ and $y = -r_\mathcal{J}^T$
   \STATE Set $P = I-\tau_{(\mathcal{J})}^{(p)}$ using \eqref{eq:tau}
   \STATE Accept $X_{t+1} = j_u$ with probability $P_{nj_u}$
   \STATE Undo relabeling; set $t = t+1$
   \UNTIL{$t = T$}
   \STATE {\bfseries Output:} $\{X_t\}_{t=0}^T \sim p^{\times (T+1)}$ (approximately)
   \end{algorithmic}
\end{algorithm}

\section{\label{sec:Remarks}Remarks}

Aside from providing a framework that unifies several different MCMC algorithms, our perspective has uncovered the apparently new HOPS algorithm of \S \ref{sec:LP}, which may enhance existing MCMC techniques specifically tailored for parallel computation, as in \cite{conrad2018expensive}. In particular, the Bayesian approach to inverse problems detailed in 
\cite{dashti2016inverse} offers fertile ground for useful applications.

As we have already indicated, the present paper is agnostic with respect to proposals, and focuses on acceptance mechanisms. However, the proposal arguably plays a more important role in practice than the acceptance mechanism, particularly for differentiable distributions. In any practical application, a stateful and/or problem-specific proposal mechanism with joint structure would likely confer significant additional power to our approach, though we leave investigations along these lines open for now (one possibility is suggested by particle MTMS algorithms as in \cite{martino2014} and exploiting tensor product structure in transition matrices and $\langle p \rangle$). It is also tempting to try to incorporate some limited proposal mechanism into the objective of \eqref{eq:LP}, but it is not clear how to usefully do this in general. We used the SK spin glass to illustrate our ideas precisely because its highly disordered structure (and discrete state space) are suited for separating concerns about proposals and acceptance.

\begin{figure}[htbp]
\includegraphics[trim = 10mm 0mm 10mm 0mm, clip, width=\columnwidth,keepaspectratio]{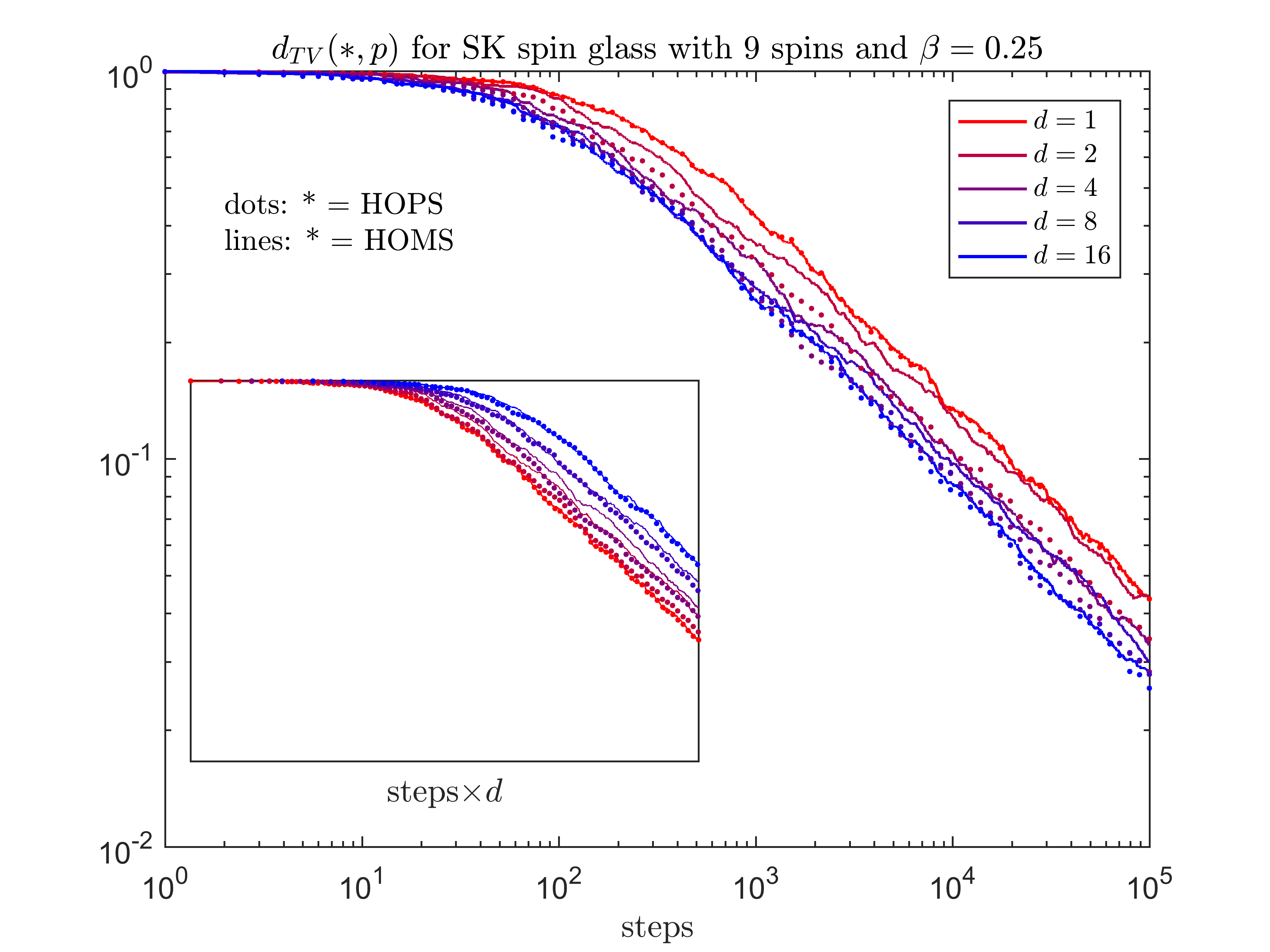}
\caption{ \label{fig:SK_9spins_b025_LPcomparison} Total variation distance between the HOPS/HOMS with proposal sets $\mathcal{J}$ (elements sampled uniformly without replacement) of varying sizes $d$ and \eqref{eq:SK} with 9 spins and $\beta = 1/4$. Inset: same data and window, with horizontal axis normalized by $d$.
} 
\end{figure} %

\begin{figure}[htbp]
\includegraphics[trim = 10mm 0mm 10mm 0mm, clip, width=\columnwidth,keepaspectratio]{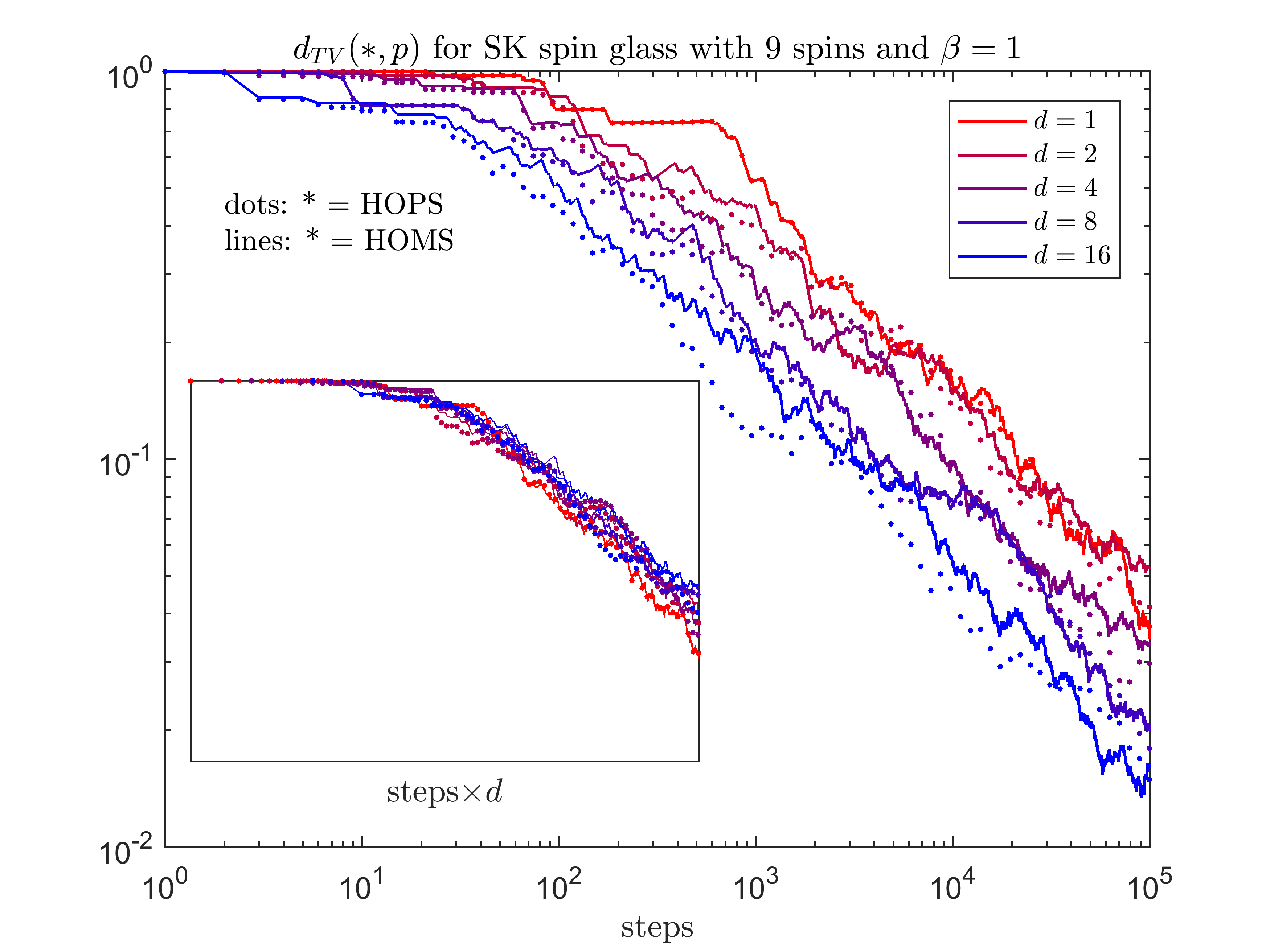}
\caption{ \label{fig:SK_9spins_b100_LPcomparison} As in Figure \ref{fig:SK_9spins_b025_LPcomparison} with $\beta = 1$.
} 
\end{figure} %

It would obviously be interesting to extend the considerations of this paper to continuous variables. However, this seems to require a much more technical treatment, as infinite-dimensional Lie theory, distributions, etc. would inevitably arise at least in principle. We leave this for future work. In a complementary vein, it would be interesting to see if the full construction of \cite{delmas2009does} can be recovered from considerations of symmetry/Lie theory alone.

While Barker and Metropolis samplers are reversible, it is not clear if the HOPS is, though \cite{bierkens2016nonreversible} points out ways to transform reversible kernels into irreversible ones and \emph{vice versa}.

We note that is possible to produce transiton matrices (as it turns out, even in closed form) in which the $n$th row is nonnegative but other rows have negative entries. It is not immediately clear if using such a matrix inevitably poisons a MCMC algorithm. Though our experiments in this direction were not encouraging, we have not found a compelling argument that rules out the use of such matrices. 

It is tempting to try to sample from the vertices of the polytope $\langle p \rangle^+$. 
However, (even approximately) uniformly sampling vertices of a polytope is $\mathbf{NP}$-hard by Theorem 1 of \cite{khachiyan2001transversal}; see also \cite{khachiyan2008vertices}.

\subsubsection*{Acknowledgements}

We thank BAE Systems FAST Labs for its support; Carlo Beenakker, Jun Liu, Lorenzo Najt, Allyson O'Brien, and Daniel Zwillinger for helpful comments; and reviewers for their careful efforts, especially for bringing our attention to \cite{delmas2009does}.

\appendix

\section{\label{sec:Proofs}Proofs}

\begin{proof}[Proof of Lemma \ref{lem:STOalgebra}.] $e_{(j,k)} e_{(\ell, m)} = (\delta_{k\ell}-\delta_{n\ell}) e_{(j,m)}$. Considering $j \leftrightarrow \ell, k \leftrightarrow m$, we are done.
\end{proof}

\begin{proof}[Proof of Lemma \ref{lem:power}.]
Using the rightmost expression in \eqref{eq:GENbasis} and using $j,k,\ell,m \ne n$ to simplify the product of the innermost two factors, we have that  
\begin{equation}
\label{eq:GENproduct}
e_{(j,k)}^{(p)} e_{(\ell,m)}^{(p)} = \left ( \delta_{k \ell} + r_\ell \right ) e_{(j,m)}^{(p)}. \nonumber
\end{equation}
Taking $j = \ell$ and $k = m$ establishes the result for $i \le 2$. The general case follows by induction on $i$. 
\end{proof}

\begin{proof}[Proof of Theorem \ref{thm:GENbasis}.]
Note that
\begin{equation}
p e_{(j,k)}^{(p)} = \left ( p_j - r_j p_n \right ) \left ( e_k^T-e_n^T \right ) \equiv 0. \nonumber
\end{equation} 
Furthermore, linear independence and the commutation relations are obvious, so it suffices to show that $\exp t e_{(j,k)}^{(p)} \in \langle p \rangle$ for all $t \in \mathbb{R}$. By Lemma \ref{lem:power}, 
\begin{align}
\label{eq:GENexp}
\exp t e_{(j,k)}^{(p)} & = & I + e_{(j,k)}^{(p)} \sum_{i = 1}^\infty \frac{t^i \left ( \delta_{jk} + r_j \right )^{i-1}}{i!} \nonumber \\
& = & I + \frac{e^{t(\delta_{jk} + r_j)} - 1}{\delta_{jk} + r_j} e_{(j,k)}^{(p)}. 
\end{align}
\end{proof}

\begin{proof}[Proof of Lemma \ref{lem:pos}.]
By hypothesis and \eqref{eq:GENbasis}, $-\sum_j t_j e_{(j,j)}^{(p)}$ has nonpositive diagonal entries and nonnegative off-diagonal entries (i.e., it is a generator matrix for a continuous-time Markov process); the result follows.
\end{proof}

\begin{proof}[Proof of Lemma \ref{lem:abcX}.]
\begin{align}
\alpha_{(\mathcal{J})}^{(p)} \beta_{(\mathcal{J})}^{(p)} & = \sum_{u,v,w,x} \alpha_{j_u j_v} \beta_{j_w j_x} e_{(j_u,j_v)}^{(p)} e_{(j_w,j_x)}^{(p)}  \nonumber \\
& = \sum_{u,v,w,x} \alpha_{j_u j_v} \left ( \delta_{j_v j_w} + r_{j_w} \right ) \beta_{j_w j_x} e_{(j_u,j_x)}^{(p)} \nonumber \\
& = \sum_{u,x} \left ( \alpha_{(\mathcal{J})} (I+1r_{(\mathcal{J})}) \beta_{(\mathcal{J})} \right )_{ux} e_{(j_u,j_x)}^{(p)}. \nonumber 
\end{align}
where the second equality follows from \eqref{eq:GENproduct} and the third from bookkeeping.
\end{proof}

\begin{proof}[Proof of Theorem \ref{thm:dxd}.]
The Sherman-Morrison formula (see \cite{horn2013matrix}) gives that
\begin{equation}
\omega (I+1r_{(\mathcal{J})})^{-1} = \omega \left ( I - \frac{1}{1+r_{(\mathcal{J})}1}1r_{(\mathcal{J})} \right ) \nonumber
\end{equation}
and the elements of this matrix are precisely the coefficients in \eqref{eq:dxdA}. Using the notation of Lemma \ref{lem:abcX}, we can therefore rewrite \eqref{eq:dxdA} as 
\begin{equation}
A_{(\mathcal{J})}^{(p;\omega)} = \left ( \omega (I+1r_{(\mathcal{J})})^{-1} \right )_{(\mathcal{J})}^{(p)}, \nonumber
\end{equation}
whereupon invoking the lemma itself yields $\left ( A_{(\mathcal{J})}^{(p;\omega)} \right )^{i+1} = \omega^i A_{(\mathcal{J})}^{(p;\omega)}$ for $i \in \mathbb{N}$. The result now follows similarly to Theorem \ref{thm:GENbasis}.
\end{proof}

\begin{proof}[Proof of Lemma \ref{lem:metropolis}.]
Writing $A \equiv A_{(\mathcal{J})}^{(p;\omega)}$ here for clarity, the result follows from three elementary observations: $\Delta ( A ) \ge 0$, $\max \Delta \left ( A \right ) > 0$, and $A - \Delta \left ( \Delta \left ( A \right ) \right ) \le 0$.
\end{proof}

\bibliography{LieMCMCbib}
\bibliographystyle{abbrvnat}

\end{document}